\documentclass{amsart}
\usepackage[left=1.25in, right=1.25in, top=1.25in, bottom=1.25in]{geometry}
\usepackage{layout}

\usepackage{amssymb, amsfonts,amsthm,amsmath,latexsym,mathrsfs}
\usepackage[all]{xy}

\def\bt{\begin{thm}}
\def\et{\end{thm}}
\def\bl{\begin{lem}}
\def\el{\end{lem}}
\def\bd{\begin{defi}}
\def\ed{\end{defi}}
\def\bc{\begin{cor}}
\def\ec{\end{cor}}
\def\bp{\begin{proof}}
\def\ep{\end{proof}}
\def\br{\begin{rem}}
\def\er{\end{rem}}

\newtheorem{thm}{Theorem}[section]
\newtheorem{prop}[thm]{Proposition}
\newtheorem{lem}[thm]{Lemma}

\newtheorem{rem}[thm]{Remark}
\newtheorem{cor}[thm]{Corollary}

\numberwithin{equation}{section}

\newcommand{\pk}{\Bbb{P}^k}
\newcommand{\pN}{\Bbb{P}^N}

\newcommand{\Hd}{\mathcal{H}_d}
\newcommand{\la}{\langle}
\newcommand{\ra}{\rangle}
\newcommand{\U}{\mathscr{U}}
\newcommand{\Cp}{\mathscr{C}_p}
\newcommand{\Cpk}{\mathscr{C}_{k-p+1}}

\newcommand{\omp}{\omega^p}
\newcommand{\ompk}{\omega^{k-p+1}}
\newcommand{\p}{\Bbb{P}}

\title[Random Holomorphic Endomorphisms ]{Ergodic Properties of \\ Random Holomorphic Endomorphisms of $\Bbb{P}^k$}

\author{Turgay Bayraktar}
\date{\today}
\address{Mathematics Department, Indiana University 47405 IN, USA}
\email{tbayrakt@indiana.edu}

\begin{document}
\begin{abstract}
We study ergodic properties of compositions of holomorphic endomorphisms of the complex projective space chosen independently at random according to some probability distribution. Along the way, we construct positive closed currents which have good invariance and convergence properties. We provide a sufficient condition for these currents to have H\"{o}lder continuous quasi-potentials. We also prove central limit theorem for d.s.h and H\"{o}lder continuous observables.
\end{abstract}

\maketitle
\section{Introduction}
 Let $f:\pk \to \pk$ be a holomorphic map of algebraic degree $d\geq2$ and $\omega_{FS}$ denote the Fubini-Study form on $\pk$ normalized by $\int\omega_{FS}^k=1.$ Dynamical \textit{Green current} $T_f$ of $f$ is defined to be the weak limit of the sequence of smooth forms $\{d^{-n}(f^n)^*\omega_{FS}\}$ (\cite{Br,HP,FS2}). Green currents play an important role in the dynamical study of holomorphic endomorphisms of the projective space \cite{FS2,Si}. The current $T_f$ has H\"{o}lder continuous quasi-potentials, hence by Bedford-Taylor theory the exterior products $$T_f^p=T_f\wedge\dots \wedge T_f$$ are also well-defined and dynamically interesting currents. In particular, the top degree intersection $\mu_f=T_f^k$ yields the unique $f$-invariant measure of maximal entropy (\cite{Lu,Si,BrDu2}). \\ \indent
 Recall that the set of rational endomorphisms $f:\pk\to \pk$ with fixed algebraic degree $d$ can be identified with $\pN$ where $N=(k+1){{d+k}\choose{d}}-1.$ We denote the set of holomorphic parameters in $\pN$ by $\Hd.$ It is well know that the complement of this set $\mathcal{M}:=\pN\backslash\Hd$ is an irreducible hypersurface \cite{GKZ}. We consider $\pN$ as a metric space furnished with the Fubini-Study metric. We let $m$ denote a Borel probability measure on $\pN$ and assume throughout the paper that a rational endomorphism $f\in\pN$ is holomorphic with probability one. 
 By a \textit{random holomorphic endomorphism} we mean a $\pN$-valued random variable with distribution $m$.

 
  In this paper, we consider the following canonical construction (see for instance \cite{Kifer}):  Let $\Omega=\prod_{i=0}^{\infty}\pN$ is the product of copies of $\pN$ endowed with the product $\sigma$-algebra and the probability measure $\p$ which is the product measure generated by finite dimensional probabilities. To avoid measurability problems, throughout this paper we assume that all probability spaces are complete and with some abuse of notation we call the completed Borel $\sigma$-algebra still as Borel algebra.\\ \indent
   In the sequel we assume that $d\geq 2.$ The elements $\lambda \in \Omega$ are sequences of rational maps
 $$\lambda=(f_0,f_1,\dots)\ \text{with} \ \lambda(n)=f_n:\pk\to \pk$$ 
of degree $d\geq2.$ We let $X_n:\Omega\to \pN$ denote the projection onto the $n^{th}$ coordinate that is
\begin{equation}\label{var}
 X_n(\lambda)=\lambda(n).
 \end{equation} 
 Note that $X_n$'s are identically distributed independent $\pN$-valued random variables with distribution $m$. We also define the unilateral shift operator \\
 $$\theta:\Omega\to \Omega$$ 
  $$(\theta\lambda)(n)=\lambda(n+1)\ \text{for all}\ n\geq 0.$$ 
It follows that the measure $\p$ is $\theta$-invariant and ergodic. A natural skew product on $X:=\Omega\times\pk$ is defined by
$$\tau:X\to X$$
$$(\lambda,x)\to (\theta(\lambda), X_0(\lambda) x)$$
note that 
$$\tau^n(\lambda,x)=(\theta^n(\lambda), F_{\lambda,n}(x))$$
where $F_{\lambda,n}:=f_{n-1}\circ \dots \circ f_{1} \circ f_{0}:\pk\to \pk$ is a rational map of algebraic degree $\leq d^n.$ We remark that the results in this paper do not depend on the specific choice of the random variables (\ref{var}) but their distribution $m.$ Note that if $m$ is a Dirac mass supported at $f\in \Hd$ then the deterministic case emerges. 
\\ \indent
Our first result indicates that the sequence of pull-backs of a smooth form by a random sequence of holormophic maps is equidistributed with a positive closed current.
  \begin{thm} \label{th}
 There exists a $\theta$-invariant set $\mathscr{A}\subset \Omega$ of probability one such that for $1\leq p\leq k$ and every $\lambda\in \mathscr{A}$ the sequence $\{d^{-pn} F_{\lambda,n}^*\omega_{FS}^p\}$ converges in the sense of currents to a positive closed bidegree $(p,p)$ current $T_p({\lambda})$ satisfying 
\begin{equation} \label{inv}
f^*_{0}(T_p({\theta(\lambda)}))=d^pT_p({\lambda}).
\end{equation}
Furthermore, if 
$$ \log dist(\cdot,\mathcal{M}) \in L^1_m(\pN)$$ 
then with probability one the current $T_p(\lambda)$ has H\"{o}lder continuous super-potentials.
\end{thm}
Random iteration of perturbation of holomorphic maps was studied in \cite{FW} (see also \cite{DS1,Peters,DS3} for the non-autonomous setting). A local version of Theorem \ref{th} was proved in \cite{FW} for $p=1$ or $k$ when $m$ is the Lebesgue measure. More recently, dynamics of fibered rational maps has been studied in \cite{J,J1,Sumi, dTh1,dTh2}. We remark that the current $T_1(\lambda)$ was previously obtained by de Th\'elin \cite{dTh1} in the setting of fibered rational maps under the assumption that the function $\log dist(\cdot,\mathcal{M})\in L^1_m(\pN)$. He also proves that $T_{1}(\lambda)$ has continuous quasi-potentials when $\log dist(\cdot,\mathcal{M})$ is integrable and obtains $T_p(\lambda)$ as an exterior product $T_1(\lambda)\wedge\dots \wedge T_1(\lambda)$. 
The novelty here is that we construct $T_p(\lambda)$ directly without assuming integrability of $\log dist(\cdot,\mathcal{M}).$  We use the super-potentials of Dinh and Sibony and quantitative estimates for resolution of $\partial\overline{\partial}$-equations \cite{DS11,GiS}.  
 Moreover, we prove that integrability of $\log dist(\cdot,\mathcal{M})$ provides H\"{o}lder continuity of super-potentials of $T_p(\lambda)$ with probability one. 
Finally, we remark that Theorem \ref{th} and its consequences can be extended to the setting of random dynamical systems of holomorphic endomorphisms (cf. \cite{J1,dTh1}). 
 
 Next, we provide an application of Theorem \ref{th} to the value distribution theory. We let $\Bbb{G}(p,k)$ denote the Grassmannian of projective-linear subspaces of codimension $p$ in $\pk.$ Note that $\Bbb{G}(k,k)=\pk.$ The following result is a direct consequence of \cite[Theorem 1.2]{RS}:
\begin{cor}
For $\p$-almost every $\lambda\in \Omega$ there exists a pluripolar set $\mathcal{E}_{\lambda} \subset \Bbb{G}(p,k)$ such that 
$$\frac{1}{d^{pn}}(F_{\lambda,n})^*[W] \to T_p(\lambda)$$
in the sense of currents as $n\to \infty$ for every $W\in \Bbb{G}(p,k)\backslash \mathcal{E}_{\lambda}.$
\end{cor}

 For $p=k$ and $\lambda\in \mathscr{A}$ each $T_k(\lambda)$ is a Borel probability measure on $\pk$ and we can define a probability measure $\mu$  on $\Omega\times \pk$ with product $\sigma$-algebra $\mathscr{B}$ whose action on a continuous function $\phi:\Omega\times \pk \to \Bbb{R}$ is given by
 $$\langle \mu,\phi\rangle:=\int_{\Omega}\langle T_k(\lambda), \phi(\lambda,\cdot)\rangle d\p(\lambda). $$ 
It follows from Theorem \ref{th} that the measure $\mu$ is well-defined and $\tau$-invariant. 
In the sequel, we consider some ergodic properties of the dynamical system $(X,\mathscr{B},\tau,\mu).$ \\ \indent
Recall that a \textit{quasi-plurisubharmonic} (qpsh for short) function is an $L^1(\pk)$ function which can be locally written as difference of a plurisubharmonic function and a smooth function. A \textit{d.s.h} function is equal to difference of two qpsh functions outside of a pluripolar set. In particular, smooth functions are dsh. The class of dsh functions was introduced by Dinh and Sibony; they are useful for the study of equidistribution problems in complex dynamics (see \cite{DS3} for instance). One can define a norm on the set of dsh functions $DSH(\pk)$ (see section \ref{dsh} for details). For a function $\psi:\pk\to \Bbb{R}$ we denote $\tilde{\psi}=\psi\circ \pi$ where $\pi:X\to \pk$ is the projection on the second factor. Next, we prove that $(X,\mathscr{B},\tau,\mu)$ has exponential decay of correlations for d.s.h (respectively H\"older continuous) observables. We remark that these strong mixing properties requires a better control on $\log dist(\cdot,\mathcal{M})$ (cf. Remark \ref{uniform} and Proposition \ref{mix}). In particular, if $\log dist(\cdot,\mathcal{M})$ is bounded, i.e. the support of $m$ is contained in $\Hd,$ we obtain exponential decay of correlations: 
  \begin{thm}\label{mixx}
 If supp$(m)\subset \Hd$ then there exists $C>0$ such that
 \begin{equation}\label{strongmix}
|\int_X (\varphi \circ \tau^n) \tilde{\psi}\ d\mu-\int_X \varphi d\mu \int_X \tilde{\psi} d\mu\ |\leq C d^{-n}||\varphi||_{L^p(X)}\ ||\psi||_{DSH(\pk)}
\end{equation}
for $n\geq 0,$ $\varphi\in L^{p}(X)$ with $p>1$ and $\psi \in DSH(\pk).$
  \end{thm}
In the special case, $m=\delta_f$ for $f\in \Hd$ we recover the corresponding result of \cite{DNS}. Next, we focus on some stochastic properties of the invariant measure measure $\mu.$ We say that a function $\psi:\pk\to\Bbb{R}$ is a coboundary if $\tilde{\psi}= h \circ \tau- h$ for some $h\in L_{\mu}^2(X).$ We prove central limit theorem (CLT for short) for dsh and H\"{o}lder continuous observables. 

\begin{thm}\label{CLT}
Assume that supp$(m)\subset \Hd.$ If $\psi:\pk\to\Bbb{R}$ is H\"{o}lder continuous or dsh which is not a coboundary such that $\la\mu,\psi\circ \pi\ra=0$ then $\tilde{\psi}=\psi\circ \pi$ satisfies CLT. That is, for every interval $I\subset \Bbb{R}$  
$$\lim_{n\to \infty}\mu\big\{(\lambda,x): \frac{1}{\sqrt{n}}\sum_{j=0}^{n-1}\psi(F_{\lambda,j}(x))\in I\big\}=\frac{1}{\sqrt{2\pi}\sigma}\int_I\exp(-\frac{x^2}{2\sigma^2})dx$$
where $\sigma>0$ given by $$\sigma^2=\lim_{n\to \infty}\frac{1}{n}\int_X( \sum_{j=0}^{n-1}\psi\circ \tau^j)^2d\mu.$$ 
\end{thm}
In the deterministic case, by means of different methods CLT was obtained for H\"{o}lder observables in \cite{ClB,DS06,Dup,DNS} and for dsh observables in \cite{DNS}. However, the published version of \cite{ClB} contains a gap; later the authors proposed another version. Here, we follow the strategy developed by \cite{DNS}; namely, we use the strong mixing property (\ref{strongmix}) and apply Gordin's method \cite{Gordin} to derive CLT.  

Finally, we consider a time homogenous Markov chain associated with the pre-images of random holomorphic maps. We define the transition probability by 
$$P:\pk\times\mathcal{B}\to [0,1]$$
$$P(x,G) :=\int_{\pN}\mathcal{L}_f(\chi_G)(x)dm(f)$$
where $\mathcal{B}$ denotes the Borel algebra on $\pk,$ $\chi_G$ denotes the indicator of $G$ and $$\mathcal{L}_f\chi_G(x):=d^{-k}\sum_{f(y)=x}\chi_G(y)$$ is the transfer operator.\\ \indent
 Let $Y$ denote the infinite product space $Y:=\prod_{n=1}^{\infty}\pk$ endowed with the product algebra $\mathcal{B^{\otimes\Bbb{N}}}$ and $\vartheta:Y\to Y$ be the unilateral shift operator. Given an initial distribution $\nu$ on the state space $(\pk,\mathcal{B}),$ we define $\Bbb{P}_{\nu}$ to be the product measure on $Y$ generated by $\nu.$ We let  $Z_0$ be a $\pk$-valued random variable whose distribution is $\nu$ that is 
$$\Bbb{P}_{\nu}[Z_0\in G]=\nu(G)$$ 
for every Borel set $G\subset \pk.$ Then we define the random variables
$$Z_n:Y\to \pk$$
$$Z_n(y):=Z_0\circ\vartheta^n(y).$$ 
The sequence $(Z_n)_{n\geq 0}$ induces a time homogenous Markov chain with state space $(\pk,\mathcal{B})$ and transition probability is $P$ such that its law $\p_{\nu}$ satisfies
$$\p_{\nu}[Z_{n+1}\in G\ |\ Z_n=x]=P(x,G)\ \ \text{and} \ \ \p_{\nu}[Z_0\in G]=\nu(G).$$ 
Now, we let $\nu:=\pi_*\mu$ where $\pi:X\to \pk$ is the projection on the second factor and $\mu$ is as above. It follows that the probability measure $\nu$ is $P$-invariant and ergodic (Proposition \ref{station}), hence, $(Z_n)_{n\geq0}$ is stationary under $\p_{\nu}.$
We say that $\psi\in L^2_{\nu}(\pk)$ is a coboundary for the Markov chain $(Z_n)_{n\geq0}$ if $\int_{\pk}\psi^2-(P\psi)^2d\nu=0.$ We prove CLT for the Markov chain $(Z_n)_{n\geq0}$ with initial distribution $\nu$ for dsh and H\"older continuous observables.
\begin{thm}\label{CLTM}
If supp$(m)\subset \Hd$ then every H\"{o}lder continuous or dsh function $\psi:\pk\to\Bbb{R}$ which is not a coboundary such that $\la \nu,\psi\ra=0$ satisfies CLT for the Markov chain $(Z_n)_{n\geq0}.$ That is, for every interval $I\subset \Bbb{R}$  
$$\lim_{n\to \infty}\p_{\nu}\big\{y\in Y: \frac{1}{\sqrt{n}}\sum_{j=0}^{n-1}\psi(Z_j(y))\in I\big\}=\frac{1}{\sqrt{2\pi}\sigma^2}\int_I\exp(-\frac{x^2}{2\sigma^4})dx $$
where $\sigma^2=\int_{\pk}\psi^2-(P\psi)^2d\nu.$
\end{thm}
\section*{Acknowledgement}
I would like to thank Professor Bernard Shiffman for many valuable conversations on the content of this work. I am also grateful to the referees for careful reading of the manuscript. Their suggestions improved the exposition of the paper.  
\section{Background}
\subsection{Super-potentials of positive closed currents}
Let $\pk$ denote the complex projective space and $\omega$ be the Fubini-Study form on $\pk$ normalized by $$\int_{\pk}\omega^k=1.$$ We denote the space of smooth $(p,q)$ forms on $\pk$ by $\mathcal{D}_{p,q}$ and let $\mathcal{D}^{p,q}=(\mathcal{D}_{k-p,k-q})'$ denote the set of bidegree $(p,q)$ currents. We say that a $(p,p)$ form $\Phi$ is (strongly) positive if at every point it can be written as a linear combination of forms of type $i\alpha_1\wedge\overline{\alpha}_1\wedge \dots \wedge i \alpha_p\wedge \overline{\alpha}_p$ where each $\alpha_i\in \mathcal{D}_{1,0}.$ In particular, a positive $(k,k)$ form is a product of a volume form and a positive function. We say that a  $(p,p)$ form $\Phi$ is weakly positive if $\Phi\wedge\varphi$ is a volume form for every positive form $\varphi\in \mathcal{D}_{k-p,k-p}.$ We say that $\Phi$ is a negative $(p,p)$ form if $-\Phi$ is positive.\\ \indent
A $(p,p)$ current $T$ is called (strongly) positive if $T\wedge \varphi$ is a positive measure for every weakly positive form $\varphi\in \mathcal{D}_{k-p,k-p}.$ A $(p,p)$ current $T$ is said to be negative if $-T$ is positive. We say that $T$ is closed if $dT=0$ in the sense of distributions. The mass of a positive closed $(p,p)$ current $T$ is defined by $\|T\|:=\int_{\pk}T\wedge\omega^{k-p}.$ We denote the set of all positive closed bidegree $(p,p)$ currents of mass one by $\Cp$ endowed with the weak topology of currents. The later is a compact convex set. We refer the reader to the manuscript \cite{DemBook} for basic properties of positive closed currents. \\ \indent
For a current $T\in \Cp,$ we denote its action on a smooth form $\Phi$ by $\la T,\Phi\ra.$ For a smooth $(p,q)$ form $\Phi$ denote by $\|\Phi\|_{\mathscr{C}^{\alpha}}$ the sum of $\mathscr{C}^{\alpha}$-norms of the coefficients in a fixed atlas. Following, \cite{DS11} for $\alpha>0$ we define a distance function on $\Cp$ by
$$dist_{\alpha}(R,R'):=\sup_{\|\Phi\|_{\mathscr{C}^{\alpha}}\leq 1}|\la R-R',\Phi\ra|$$
where $\Phi$ is a smooth $(k-p,k-p)$ form on $\pk.$ It follows from interpolation theory between Banach spaces \cite{Triebel} that
$$dist_{\beta}\leq dist_{\alpha}\leq C_{\alpha\beta}[dist_{\beta}]^{\frac{\alpha}{\beta}}$$
for $0<\alpha\leq \beta<\infty$ (see \cite[Lem. 2.1.2]{DS11} for the proof).  Moreover, for $\alpha\geq1$ 
$$dist_{\alpha}(\delta_a,\delta_b)\simeq\|a-b\|$$
where $\delta_a$ denotes the Dirac mass at $a$ and $\|a-b\|$ denotes the distance on $\pk$ induced by the Fubini-Study metric. We also remark that for $\alpha>0$ topology induced by $dist_{\alpha}$ coincides with the weak topology on $\Cp$ \cite[Prop. 2.1.4]{DS11}.
\\ \indent
 Let $T\in \Cp$ with $p\geq1$ then a $(p-1,p-1)$ current $U$ is called a \textit{quasi-potential} of $T$ if it satisfies the equation
 \begin{equation}\label{potential}
 T=\omega^p+dd^cU
 \end{equation}
 where $d=\partial+\overline{\partial}$ and $d^c:=\frac{i}{2\pi}(\overline{\partial}-\partial)$.
  In particular, if $p=1$ a quasi-potential is nothing but a qpsh function. Note that two qpsh functions satisfying (\ref{potential}) differ by a constant. When $p>1$ the quasi-potentials differ by $dd^c$-closed currents. The quantity $\la U,\omega^{k-p+1}\ra$ is called the \textit{mean} of $U.$ The following result provides solutions to (\ref{potential}) with quantitative estimates.
  \begin{thm}\cite{DS11}\label{DSq}
  Let $T\in \Cp$ then there exists a negative quasi-potential $U$ of $T$ which depends linearly on $T$ such that the mean of $U$ satisfies  $$| \la U,\ompk \ra|\leq C$$  
  where $C>0$ independent of $T\in \Cp.$
  \end{thm}
The quasi-potential $U$ is obtained in \cite{DS11} by using a kernel which solves $dd^c$-equation for the diagonal $\Delta$ of $\pk\times\pk$ (see also \cite{GiS}). More precisely, for $T\in\Cp$
$$U(z)=\int_{z\not=\zeta}T(\zeta)\wedge K(z,\zeta)$$
and the kernel $K(z,\zeta)$ has tame singularities in the sense that
\begin{equation}\label{DSes}
\|K(z,\zeta)\|_{\infty} \lesssim -\text{dist}(z,\zeta)^{2(1-p)}\log \text{dist}(z,\zeta)\ \ \text{and}\ \ \|\nabla K(z,\zeta)\|_{\infty} \lesssim \text{dist}(z,\zeta)^{1-2p}
\end{equation}
where $\|\nabla K\|_{\infty}$ denotes the sum $\sum_j|\nabla K_j|$ and $K_j$'s are the coefficients of $K$ for a fixed atlas of $\pk\times\pk.$ \\ \indent
Super-potentials of positive closed currents were introduced by Dinh and Sibony \cite{DS11} which extends the notion of quasi-potential defined for the positive closed bidegree $(1,1)$ currents. If $T$ is a smooth form in $\Cp,$ \textit{super-potential} of $T$ of mean $m$ is  defined by
$$\U_T:\Cpk \to \Bbb{R}\cup\{-\infty\}$$  
\begin{equation}\label{super}
\U_T(R)=\la U_T,R\ra
\end{equation}
where $U_T$ is a quasi-potential of $T$ of mean $m.$ Then it follows that (see \cite[Lemma 3.1.1]{DS11})
$$\U_T(R)=\la T,U_R\ra$$
where $U_R$ is a quasi-potential of $R$ of mean $m.$ In particular, the definition of $\U_T$ in (\ref{super}) is independent of the choice of $U_T$ of mean $m$. Note that super-potential of $T$ of mean $m'$ is given by $\U_T+m'-m.$ More generally, for an arbitrary current $T\in\Cp$ super-potential of $T$ is defined by $\U_T(R)$ on smooth forms $R\in \Cpk$ as in (\ref{super}) where $U_R$ is smooth. Then the definition of super-potential can be extended to a function on $\Cpk$ with values in $\Bbb{R}\cup \{-\infty\}$ by approximation (see \cite[Proposition 3.1.6]{DS11}). In the sequel, for $T\in \Cp$ we denote $\|\U_T\|_{\infty}=\sup_{R\in\Cpk}|\U_T(R)|.$ Note that the later $\sup$ is finite if $T\in \Cp$ is smooth. Finally, we remark that super-potentials determine the currents:
\begin{prop}\cite{DS11}\label{determine}
Let $S,S'$ be currents in $\Cp$ with super-potentials $\U_S,\U_{S'}.$ If $\U_S=\U_{S'}$ on smooth forms in $\Cpk$ then $S=S'.$
\end{prop}
 
\subsubsection{Super-potentials of pull-back and push-forward by holomorphic endomorphisms}
Let $f:\pk\to \pk$ be a holomorphic endomorphism of algebraic degree $d\geq 2.$ For $1\leq p\leq k$ the pull-back and push-forward operators 
$$L:=d^{-p}f^*:\Cp\to \Cp$$
$$\Lambda:=d^{-p+1}f_*\Cpk\to \Cpk$$
are well defined and continuous (see for instance \cite{Meo,DS07,DS11}). 
\begin{prop}\cite{DS11}\label{push-pull}
Let $S\in \Cp$ (respectively $R\in\Cpk$). We also let $\U_S$ (respectively $\U_R$) and $\U_{L(\omp)}$ (respectively $\U_{\Lambda(\ompk)}$) be super-potentials of $S$ (respectively $R$) and $L(\omp)$ (respectively $\Lambda(\ompk)$). Then the currents $L(S)$ and $\Lambda(R)$ admit super-potentials given by
$$ \U_{L(S)}=\frac{1}{d}\U_S\circ \Lambda+\U_{L(\omp)}$$
$$\U_{\Lambda(R)}=d\U_R\circ L+ \U_{\Lambda(\ompk)}.$$

\end{prop}
We refer the reader to \cite{DS11} for further properties of super-potentials.
\subsection{Moderate currents}
Following \cite{DS1,DNS}, we say that a positive closed bidegree $(p,p)$ current $T$ is \textit{moderate} if for every compact family of qpsh functions $\mathcal{K}$ there exists constants $c>0$ and $\rho>0$ such that
\begin{equation}\label{modest}
\int_{\pk}e^{-\rho \varphi}\ T\wedge \omega^{k-p}\leq c
\end{equation}
for every $\varphi\in \mathcal{K}.$ The existence of $\rho$ and $c$ in $(\ref{modest})$ is equivalent to the existence of $\rho'>0$ and $c'>0$ satisfying 
\begin{equation}\label{modest2}
T\wedge\omega^{k-p}\{z\in \pk:\varphi(z)<-M\}\leq c'e^{-\rho'M}
\end{equation}
for every $M\geq 0$ and $\varphi\in \mathcal{K}.$ It follows from \cite{DNS} ( see also \cite{DN12}) that if $T\in\Cp$ has H\"{o}lder continuous super-potentials 
\begin{equation}\label{expoex}
|\U_T(R)-\U_T(R')|\leq Cdist_{\alpha}(R,R')^{\beta}
\end{equation} where $\alpha>0$ fixed then there exists constants $c',\rho'>0$ as in (\ref{modest2}) which depend only on $\pk,\mathcal{K},C$ and H\"{o}lder exponent $\beta.$ Thus, we have the following result:  
\begin{thm}\cite{DNS} \label{moderate}
Let $T\in \Cp$ admit a H\"{o}lder continuous super-potential as in (\ref{expoex}). Then the current $T$ is moderate. Moreover, the constants $c,\alpha$ in (\ref{modest}) depend only on $\pk, \mathcal{K},C$ and H\"{o}lder exponent $\beta.$ 
\end{thm}

\section{Random Green currents} \label{random}

In this section we prove Theorem \ref{th} in a slightly more general context:
\begin{thm} \label{main}
Let $\{T_n\}_{n\geq0}$ be a sequence of positive closed bidegree $(p,p)$ currents of mass one such that $||\U_{T_n}||_{\infty}=o(d^n).$ Then the sequence $\{d^{-pn}(F_{\lambda,n})^*T_n\}$ almost surely converges weakly to a positive closed bidegree $(p,p)$ current $T_p(\lambda)$ satisfying 
\begin{equation}\label{inv}
f_0^*(T_p(\theta(\lambda)))=d^pT_p(\lambda).
\end{equation}
\end{thm}

The current $T_p(\lambda)$ is called as  \textit{random Green current} associated with  the distribution $m.$ 
\begin{proof}
Let $\mathcal{A}$ be the set of $\lambda\in \Omega$  such that $d^{-pn}(F_{\lambda,n})^*T_n$ is well-defined for every $n\geq 0$ and converges weakly to a positive closed bidegree $(p,p)$ current. For fixed $\epsilon>0$ we also denote  
  $$\Hd^{\Bbb{N}}(\epsilon):=\{\lambda\in \Omega:\ dist(\lambda(j),\mathcal{M})\geq e^{-\epsilon j}\  \text{for all but finitely many}\ j\in \Bbb{N} \}\subset \Hd^{\Bbb{N}}.$$ 
Then Kolmogorov's zero-one law  and the fact that $m$ is a Borel measure with $m( \mathcal{M})=0$ implies that the set  $\Hd^{\Bbb{N}}(\epsilon)$ has probability one and clearly invariant under the shift $\theta.$ We will prove that $\Hd^{\Bbb{N}}(\epsilon)\subset \mathcal{A}$ for sufficiently small $\epsilon>0.$ Assuming $\Hd^{\Bbb{N}}(\epsilon)\subset \mathcal{A}$ for the moment, since $(\Omega, \mathscr{B},\p)$ is a complete probability space this implies that $\mathcal{A}$ is measurable and has probability one. 
Then we define
$$\mathscr{A}:=\cap_{n=0}^{\infty}\theta^n(\mathcal{A})$$
which is clearly measurable and invariant under $\theta.$ Moreover, $\mathscr{A}$ contains $\Hd^{\Bbb{N}}(\epsilon)$ hence $\mathscr{A}$ has probability one.  \\ \indent 
 Next, we prove that $\Hd^{\Bbb{N}}(\epsilon)\subset \mathcal{A}$ for small $\epsilon>0.$ In the rest of the proof, we denote $f_j:=\lambda(j)$ where $\lambda \in \Hd^{\Bbb{N}}(\epsilon)$ and denote the pull-back and push-forward operators by
  $$L_j:=\frac{1}{d^p}f_j^*:\mathscr{C}_p\to\mathscr{C}_p$$
  $$\  \Lambda_j:=\frac{1}{d^{p-1}}(f_j)_*:\Cpk\to \Cpk.$$
We also set $$\Lambda^ {j}:=\Lambda_{j-1}\circ \Lambda_{j-2} \circ \dots \circ \Lambda_{0}$$ for $j\geq 1$ with the convention that $\Lambda^{0}=id.$ 
By Theorem \ref{DSq}
there exists smooth negative $(p-1,p-1)$ currents $U_{L_{j}(\omp)}$ and $C<0$ independent of $j$ such that  
$$dd^cU_{L_{j}(\omp)}:=\frac{1}{d^p}f_j^*\omp-\omp.$$
and $$C\leq m_j:=\la U_{L_{j}(\omp)},\ompk\ra \leq0$$ for every $j\geq 0.$ Let $\U_{L_j(\omp)}$ be the super-potential of $L_j(\omp)$ of mean $m_j.$  It follows from Proposition \ref{push-pull} that 
$$\U_{\lambda,n}:=\frac{1}{d^n}\U_{T_n}\circ \Lambda^{n}+\sum_{j=0}^{n-1}\frac{1}{d^j} \U_{L_{j}(\omp)}\circ \Lambda^{j}$$ 
is a super-potential of $d^{-pn}(F_{\lambda,n})^*T_n$ on the smooth forms in $\Cpk.$ The first term converges to zero by assumption $||\U_{T_n}||_{\infty}=o(d^n)$. Hence, $\U_{\lambda,n}$ converges to 
$$\U_{T_p(\lambda)}:=\sum_{j=0}^{\infty}\frac{1}{d^j}\U_{L_{j}(\omp)}\circ \Lambda^{j}.$$ 
on the smooth forms in $\Cpk.$ By \cite[Corollary 3.2.7]{DS11}, it is enough to show that $\U_{T_p(\lambda)}$ is not identically $-\infty.$ To this end, it is sufficient to prove that the sequence of means $\U_{\lambda,n}(\ompk)$ is bounded from below.\\ \indent
 Now, since $\Lambda_j(\ompk)$ is a positive closed bidegree $(k-p+1,k-p+1)$ current of mass one, we may write it as $$\Lambda_{j}(\ompk)=\ompk+dd^cR_{j}$$ where $R_{j}$ is a negative bidegree $(k-p,k-p)$ current given by Theorem \ref{DSq} and its mean satisfies $$M\leq c_j:=\la R_j,\omp\ra \leq 0$$ for some constant $M<0$ independent of $j.$  
  
 Note that for $f_j\in \Hd$ the operator $\Lambda_j$ can be continuously extended to set of  negative bidegree $(k-p,k-p)$ currents $R$ such that $dd^cR\geq -\ompk.$ Moreover, 
 \begin{eqnarray*}
 \la \Lambda_{j}(R),\omp\ra &=& \la R, \frac{1}{d^{p-1}}f_{j}^*\omp\ra\\
&=&  d\la R,\omp\ra +d\la R,dd^cU_{L_j(\omp)}\ra\\
&=& d\la R,\omp\ra +d\la U_{L_j(\omp)},dd^cR\ra.
 \end{eqnarray*}
Then by Lemma \ref{pull} below, the norm $\|U_{L_j(\omp)}\|_{\mathscr{C}^{\alpha}}$ is bounded by $C_{\alpha} dist(f_j,\mathcal{M})^{-q}$ for some constants $C_{\alpha}>0$ and $q\geq 1$ independent of $f_j.$ This implies that there exists $C_1>0$ independent of $j$ such that
 $$ U_{L_j(\omp)}+C_1dist(f_j,\mathcal{M})^{-q}\omega^{p-1}\geq 0$$ in the sense of currents for sufficiently large $j$ say $j\geq j_{\lambda}.$ 
Hence, we infer that
 $$ \la \Lambda_{j}(R),\omp\ra\geq d\la R,\omp\ra -C_2de^{\epsilon q j}$$ where $C_2>0$ independent of $j.$ Now, writing 
$$\Lambda^n(\ompk)=\ompk+dd^cS_n$$
where $S_n=\sum^{n-2}_{i=0}\Lambda_{n-1}\circ \dots \circ \Lambda_{i+1}(R_{i})+R_{n-1}$ from above estimate we deduce that $S_n$ is a decreasing sequence of negative bidegree $(k-p,k-p)$ currents 
  such that 
  \begin{eqnarray*}
  \frac{1}{d^n}\la S_n,\omp\ra &\geq& \sum_{j=0}^{n-1}d^{-j}\la R_j,\omp\ra-d^{-n+1}\sum_{j=1}^{n-1}je^{\epsilon qj}\\
&\geq&   \frac{dM}{d-1}- O(\frac{e^{2\epsilon q(n-1)}}{d^{n-1}})
  \end{eqnarray*}
thus, choosing $0<\epsilon<\frac{1}{2q}\log d$ we see that $\frac{1}{d^n}\la S_n,\omp\ra$ is bounded. Then 
\begin{eqnarray*}
\U_{L_{j}(\omp)}\circ \Lambda^{j}(\ompk) &=& \U_{L_{j}(\omp)}(\ompk+dd^cS_{j})\\
&=&m_j+\la U_{L_{j}(\omp)},dd^cS_{j} \ra
\end{eqnarray*}
and since $U_{L_{j}(\omp)}$ is smooth we have
\begin{eqnarray} \label{e1}
\la U_{L_{j}(\omp)},dd^cS_{j} \ra &=& \la dd^c U_{L_{j}(\omp)},S_{j}\ra \\
&=& \la\frac{1}{d^p} (f_{j})^*\omp-\omp,S_{j}\ra 
\end{eqnarray}
Then from $S_{j+1}=\frac{1}{d^{p-1}}(f_j)_*S_j+R_j$ we infer that
\begin{equation}\label{e2}
\la \frac{1}{d^p}(f_j)^*\omp,S_j\ra=\frac{1}{d}\la \omp, S_{j+1}-R_j\ra
\end{equation}
Since $R_j$ is a negative current combining (\ref{e1}) and (\ref{e2}) we obtain  
 \begin{eqnarray*}
\U_{L_{j}(\omp)}\circ \Lambda^{j}(\ompk)
 &=& m_j-\frac{1}{d}\la R_j,\omp \ra + \la\frac{1}{d} S_{j+1}-S_{j},\omp\ra \\
& \geq & m_j+\la \frac{1}{d} S_{j+1}-S_{j},\omp\ra
\end{eqnarray*}
Hence,
\begin{eqnarray*}
\U_{\lambda,n}(\ompk)\geq \sum_{j=0}^{n-1}\frac{m_j}{d^j}+\frac{1}{d^n}\la S_{n},\omp\ra
\end{eqnarray*}
from $0\geq m_j \geq C$ we deduce that
$$\U_{\lambda,n}(\ompk)\geq C\frac{d}{d-1} + \frac{1}{d^n}\la S_{n},\omp\ra.$$
Since the last term is bounded the first assertion follows. \\ \indent
Note that the super-potential $\U_{\theta(\lambda),n}$ of $\{d^{-pn}(F_{\theta(\lambda),n})^*\omp\}$ satisfies
\begin{equation}\label{superpot}
\frac{1}{d}\U_{\theta(\lambda),n}\circ \Lambda_0+\U_{L_0(\omp)}=\U_{\lambda,n+1}
\end{equation} 
on smooth forms in $\Cpk,$ then (\ref{superpot}) together with Proposition \ref{determine} implies that 
$$(f_0)^*(d^{-pn}(F_{\theta(\lambda),n})^*\omp)=d^{-pn}(F_{\lambda,n+1})^*\omp.$$
Since both sequences are convergent and $\frac{1}{d^p}f_0^*$ is continuous on $\Cp$ passing to the limit we see that $$f_0^*(T_p(\theta(\lambda))=d^pT_p(\lambda).$$
\end{proof}
In the sequel, we will show that super potentials of $T_p(\lambda)$ are H\"{o}lder continuous for $1\leq p\leq k$ with respect to the $dist_{\alpha}$ for some (equivalently for all) $\alpha>0$ under the assumption that $\log dist(\cdot,\mathcal{M})\in L^1_m(\pN).$ 
 
\begin{thm}\label{holder}
Assume that $\log dist(\cdot,\mathcal{M})\in L^1_m(\pN).$ Then with probability one random Green current $T_p(\lambda)$ has H\"{o}lder continuous super-potentials.
\end{thm}

In what follows, we use the same notation as in the proof of Theorem \ref{main}. We prove several lemmas which will be useful in the proof of Theorem \ref{holder}. First, we show that with probability one the maps in the tail of $\lambda$ do not get too close to $\mathcal{M}.$ More precisely, 
\begin{lem}\label{dist} If $\log dist(\cdot,\mathcal{M})\in L^1_m(\pN)$ then the set 
$$A:=\{\lambda\in \Omega: \lim_{n\to \infty}\frac{1}{n}\sum_{j=0}^{n-1}\log dist(f_j,\mathcal{M})= \int_{\pN} \log dist(f,\mathcal{M})dm(f)\}$$
has probability one. Furthermore for every $\epsilon>0$ and for $\p$-a.e. $\lambda$ there exists $n_{\epsilon}(\lambda)$ such that  
$$dist(f_j,\mathcal{M})\geq exp(-j\epsilon)$$
for every $j\geq n_{\epsilon}.$ 
\end{lem}
\begin{proof}

We define the random variables $$\mathcal{X}_j(\lambda):=\log dist(f_j,\mathcal{M}).$$ Note that $\mathcal{X}_j$'s are independent, identically distributed sequence random variables with finite mean. 
Thus, it follows from strong law of large numbers (see \cite[Theorem 22.1]{Bil}) that  with probability one, $\frac{1}{n}\sum_{j=0}^{n-1}\mathcal{X}_j$ converges to the mean of $\mathcal{X}_0$ namely, $E(\mathcal{X}_0)=\int_{\pN}\log dist(f,\mathcal{M})dm(f).$ This proves the first assertion. \\ \indent
 To prove the second assertion, let $\epsilon>0$ be small and define
$$B_j:=\{\lambda \in \Omega: dist(f_j,\mathcal{M})< e^{-j\epsilon}\}$$
Note that $B_j$'s are independent events. Since $ \log dist(\cdot,\mathcal{M}) \in L^1_m(\pN)$ and $f_j's$ are i.i.d. for every $\epsilon>0$ the sum $\sum_{j=1}^{\infty} \p(B_j)$
converges.
Indeed, since $\p$ is the product measure we have 
\begin{eqnarray*}
\int_{\pN}-\log dist(f,\mathcal{M})dm & = & \int_0^{\infty}m\{f\in \pN:\log dist(f,\mathcal{M})<-t\}dt \\
& \geq & \sum_{j=0}^{\infty}m\{f\in \pN:\log dist(f,\mathcal{M})<-j\epsilon\} \\
& = & \sum_{j=0}^{\infty} \p\{\lambda\in \Omega:dist(f_j,\mathcal{M})<e^{-j\epsilon}\}\\
& = &  \sum_{j=0}^{\infty} \p(B_j) .
\end{eqnarray*}  
 Thus, by Borel-Cantelli Lemma  \cite{Bil} we have $$\p(\limsup_{j\to \infty} B_j)=0$$  where
$$\limsup_{j\to\infty} B_j=\cap_{j=1}^{\infty}\cup_{k=j}^{\infty}B_k.$$ 

\end{proof}
 Note that the set $A$ is invariant under the shift $\theta:\Omega\to \Omega.$ Moreover, if $\lambda\in A$ then $f_j\in \Hd$ for every $j\geq 0.$ 

Next, we observe that $\Lambda_j=\frac{1}{d^{p-1}}(f_j)_*$ is Lipschitz on $\Cpk$ with respect to the distance $dist_{\alpha}.$ In the sequel we fix $\alpha\geq 1.$  We utilize some arguments from \cite[Lemma 5.4.3]{DS11} and \cite[Proposition 3]{dTh1}:
\begin{lem}\label{push} Let $f_j\in \Hd$ then there exists constants $K>0, q\geq 1$ independent of $f_j$  such that
$$dist_{\alpha}(\Lambda_j(R),\Lambda_j(R'))\leq K dist(f_j,\mathcal{M})^{-q} dist_{\alpha}(R,R')$$
for every $R,R'\in \Cpk.$
\end{lem}
\begin{proof}
First, we prove that there exists constants $\rho_{\alpha}>0$ and $q\geq 1$ independent of $f_j$ such that
$$\|f_j^*\Phi\|_{\mathscr{C}^{\alpha}}\leq \rho_{\alpha} dist(f_j,\mathcal{M})^{-q}\|\Phi\|_{\mathscr{C}^{\alpha}}$$
for every $(p-1,p-1)$ smooth form  $\Phi$ on $\pk$ satisfying $\|\Phi\|_{\mathscr{C}^{\alpha}}\leq 1.$
 Indeed, we consider the meromorphic map 
$$\Psi:\pN\times \pk\to \pk$$  
$$\Psi(g,x)=g(x).$$
 By \cite[Lemma 2.1]{DD} there exists $C>0$ and $q_1\geq 1$ such that 
$$\|D_xf_j\|\leq \|D_{(f_j,x)}\Psi\|\leq C dist(f_j,\mathcal{M})^{-q_1}$$
for every $f_j\in \Hd$ and $x\in \pk.$\\ \indent
Note that $\|\Phi\|_{\mathscr{C}^{\alpha}}$ is the sum of $\mathscr{C}^{\alpha}$-norms of the coefficients in a fixed atlas. In local coordinates we may write $\Phi=\sum \phi_{IJ}dz_I\wedge d\overline{z_J}$.  Then
$$f_j^*\Phi=\sum \phi_{IJ}\circ f_j\ df_I\wedge d\overline{f_J}$$ and we infer that
\begin{equation}\label{fix}
\|f_j^*\Phi\|_{\mathscr{C}^{\alpha}}\leq \rho_{\alpha}dist(f_j,\mathcal{M})^{-q}\|\Phi\|_{\mathscr{C}^{\alpha}}\end{equation}
 for some $\rho_{\alpha}>0$ and $q\geq 1.$\\ \indent
Now, since
$$|\la\Lambda_j(R-R'),\Phi\ra|=d^{1-p}|\la R-R',f_j^*\Phi\ra|$$
the assertion follows.
 
\end{proof}
Note that a similar reasoning as in Lemma \ref{push} implies that $\U_{L_j(\omp)}$ is $\beta$-H\"{o}lder continuous for $0<\beta\leq 1$ with respect to the distance $dist_{\alpha}.$ More precisely: 
\begin{lem}\label{pull}
For $f_j\in \Hd$ and every  $0<\beta\leq 1$ there exists constants $r_{\alpha}>0$ and $q\geq 1$ independent of $f_j$ such that   $$|\U_{L_j(\omp)}(R)-\U_{L_j(\omp)}(R')|\leq r_{\alpha}dist(f_j,\mathcal{M})^{-q} dist_{\alpha}(R,R')^{\beta}$$
for every $R,R'\in \Cpk.$
\end{lem}
\begin{proof}
It follows from Theorem \ref{DSq} that there exists a negative smooth $(p-1,p-1)$ form $U_{L_j(\omp)}$ of mean $m_j$ such that
$$L_j(\omp)=\omega^p+dd^cU_{L_j(\omp)}.$$
the form $U_{L_j(\omp)}$ is defined by 
$$U_{L_j(\omp)}(z)=\int_{\zeta\not=z}L_j\omp\wedge K(z,\zeta)$$
where $K(z,\zeta)$ is the negative kernel on $\pk\times\pk$. Applying the argument in the proof of Lemma \ref{push} and using the estimates (\ref{DSes}) on the kernel $K(z,\zeta)$ we see that $\|U_{L_j(\omp)}\|_{\mathscr{C}^{\alpha}}$ is bounded by $r_{\alpha} dist(f_j,\mathcal{M})^{-q}$ for some constants $r_{\alpha}>0$ and $q\geq 1$ independent of $f_j$.
Now, since $\U_{L_j(\omp)}$ denotes the super-potential of $L_j(\omp)$ of mean $m_j$ we have
$$|\U_{L_j(\omp)}(R)-\U_{L_j(\omp)}(R')|=|\la U_{L_j(\omp)},R-R'\ra|$$
and the result follows.
\end{proof}
 In the sequel, we replace $A$ in Lemma \ref{dist} by $\mathscr{A}\cap A$ and with some abuse of notation still denote by $\mathscr{A}$ which is clearly $\theta$-invariant and has probability one. Now, we fix $\lambda \in \mathscr{A}$ and we will show that $T_p(\lambda)$ has H\"{o}lder continuous super-potential with respect to the $dist_{\alpha}$.  Our approach is similar to that of \cite{DS11}.
\begin{proof}[Proof of Theorem \ref{holder}]
Let $\epsilon>0$ small, by Lemma \ref{dist}  for a.e. $\lambda\in \Omega$ there exists $N=N(\lambda)$ such that $dist(f_j,\mathcal{M})\geq e^{-\epsilon j}$ for $j\geq N$. Moreover, $$\prod_{i=0}^{j-1}dist(f_j,\mathcal{M})^{-q}\lesssim M^{j}$$ where $M:=e^{-q \la m,\log dist(\cdot,\mathcal{M})\ra}.$  Let $0<\beta<\frac{\log d-\epsilon q}{\max(1,\log(KM))}$ where $K>0$ is given by Lemma \ref{push}. Recall that $$\U_{T_p(\lambda)}=\sum_{j=0}^{\infty}d^{-j}\U_{L_j(\omp)}\circ \Lambda^j$$ on smooth forms in $\Cpk$. Then applying Lemma \ref{pull} and Lemma \ref{push} respectively we obtain
\begin{eqnarray*}
|\U_{T_p(\lambda)}(R)-\U_{T_p(\lambda)}(R')| &\leq & C_1   \sum_{j=0}^{\infty}d^{-j}dist(f_j,\mathcal{M})^{-q} dist_{\alpha}(\Lambda^j(R),\Lambda^j(R'))^{\beta} \\ 
& \leq & C_1
\big( \sum_{j=0}^{N}d^{-j}dist(f_j,\mathcal{M})^{-q}(KM)^{\beta j} + \sum_{j=N+1}^{\infty}(\frac{e^{\epsilon q} (KM)^{\beta}}{d})^j\big) dist_{\alpha}(R,R')^{\beta} \\
& \leq & C_2(C_{N(\lambda)} + (\frac{e^{\epsilon q}(KM)^{\beta}}{d})^{-N}) dist_{\alpha}(R,R')^{\beta}  
\end{eqnarray*}
where $C_1,C_2>0$ constants independent of $N=N(\lambda)$. Hence, we conclude that super-potential $\U_{T_p(\lambda)}$ is $\beta$-H\"{o}lder continuous with $\beta<\frac{\log d}{\max(1,\log(KM))}$.
\end{proof}
\begin{rem}\label{uniform}
Note that for $\lambda\in \mathscr{A}$ fixed, $$\U_{T^p(\theta^n(\lambda))}=\sum_{j=n}^{\infty}d^{-j}\U_{L_j(\omp)}\circ \Lambda_{j-1}\circ \dots \circ \Lambda_n$$
for $n\geq 0.$ Thus, the argument in the proof of Theorem \ref{holder} implies that super-potentials of random Green currents satisfy a uniform H\"{o}lder estimate in the following sense
$$|\U_{T_p(\theta^n(\lambda))}(R)-\U_{T_p(\theta^n(\lambda))}(R')|\leq C_{\lambda} dist_{\alpha}(R,R')^{\beta}$$ for $n\geq 0$ where $C_{\lambda}$ does not depend on $n.$ Furthermore, if $\log dist(\cdot,\mathcal{M})$ is bounded, i.e. supp$(m)\subset \Hd$ then with probability one we can obtain the same estimate where $C_{\lambda}$ replaced with $C>0$ which does not depend on $\lambda.$
\end{rem}
 As a consequence of Theorem \ref{moderate}, we have the following uniform estimate which will be useful in the sequel. Let $\mathcal{K}$ denote a compact family of qpsh functions.
\begin{cor}\label{mod}
 The current $T_p(\lambda)$ is moderate with probability one. In particular, for $p=k$ and $\lambda\in \mathscr{A}$ there exists constants  $\rho(\lambda),c(\lambda)>0$ such that 
$$\int_{\pk}e^{-\rho\varphi}d\mu_{\theta^n(\lambda)} \leq c$$  for every $n\in \Bbb{N}$ and  $\varphi\in \mathcal{K}.$ 
\end{cor}


 
  


\section{Ergodic Properties of $(X,\mathscr{B},\mu,\tau)$} \label{dsh}
Let $X:=\Omega\times\pk$ and $\tau:X\to X$ be the skew product defined in the introduction. In the previous section we proved that there exists a Borel set $\mathscr{A}\subset \Hd^{\Bbb{N}} $ which is invariant under the shift operator and of probability one such that for every $\lambda\in \mathscr{A}$ the sequence $d^{-kn}F_{\lambda,n}^*\omega^k$ converges weakly to a probability measure which we denote by $\mu_{\lambda}.$ The measure $\mu_{\lambda}$ has H\"older continuous super-potentials with probability one. Moreover, by the invariance property (\ref{inv}) we have 
$$f_0^*\mu_{\theta(\lambda)}=d^{-k}\mu_{\lambda}$$
for every $\lambda\in \mathscr{A}.$ Furthermore, since $d^{-k}f_*f^*= id$ on $\mathscr{C}_k$ for every $f\in \Hd$ we infer that
\begin{equation}\label{pusheq}
(f_0)_*\mu_{\lambda}=\mu_{\theta(\lambda)}.
\end{equation} 
 
Let $\chi_{\mathscr{A}}$ denote the indicator function of ${\mathscr{A}}\subset \Omega.$ Note that for every $(k-p,k-p)$ test form $\Phi$ on $\pk$ the map
\begin{equation}\label{map}
\lambda\to \chi_{\mathscr{A}}(\lambda)\la T_p(\lambda),\Phi\ra
\end{equation}
is measurable. Indeed, for each $n$ we consider the map
$$\lambda\to \chi_{\mathscr{A}}(\lambda)\la d^{-pn}F_{\lambda,n}^*\omega^p,\Phi\ra$$
which is measurable since for each $\lambda\in \mathscr{A},$ the form $F_{\lambda,n}^*\omega^p$ is smooth and its coefficients depend continuously on $\lambda$ as $\lambda$ varies in $\Hd^{\Bbb{N}}.$ Now, being limit of measurable maps (\ref{map}) defines a measurable map. Therefore, 
$$\mathcal{T}_p:=\int_{\Omega}\la T_p(\lambda),\cdot\ra d\p(\lambda)$$
defines a positive closed bidegree $(p,p)$ current on $\pk$. Let
$$ \Theta_p:\Cp\to\Cp$$
$$\Theta_p(S):=d^{-p}\int_{\pN}f^*S\ dm(f)$$
 It follows form (\ref{inv}) and $\theta$-invariance of $\p$ that the current $\mathcal{T}_p$ is invariant under the operator $\Theta_p.$ That is
 $$\Theta_p(\mathcal{T}_p)=\mathcal{T}_p.$$
We denote the top degree current by $\nu:=\mathcal{T}_k$ which is a Borel probability measure on $\pk$.\\ \indent
  Now, by above reasoning we may also define a probability measure on $X$ by 
  $$\la\mu,\varphi\ra=\int_{\Omega}\la\mu_{\lambda},\varphi\ra d\p(\lambda)$$ where $\varphi:X\to \Bbb{R}$ is a continuous function. Note that $\pi_*\mu=\nu$ where $\pi:X\to \pk$ is the projection on the second factor. Moreover, the measure $\mu$ is $\tau$-invariant. Indeed, by (\ref{pusheq}) and $\theta$-invariance of $\p$ we have 
\begin{eqnarray*}
\la\tau_*\mu,\varphi\ra=\la\mu,\varphi\circ \tau\ra &= & \int_{\Omega}\la\mu_{\lambda}(dx),\varphi(\theta(\lambda),f_0(x))\ra d\p(\lambda)\\
& = & \int_{\Omega}\la\mu_{\theta(\lambda)}(dx),\varphi(\theta(\lambda),x)\ra d\p(\lambda)\\
& = & \int_{\Omega} \la \mu_{\lambda},\varphi\ra d\p(\lambda)\\
& = & \la \mu,\varphi\ra
\end{eqnarray*}
  
  Furthermore, since $(\Omega,\p,\theta)$ is mixing, it is classical that $(X,\mu,\tau)$ is also mixing (see \cite[Proposition 4.1]{J1}). In section \ref{expo} we will show that the dynamical system $(X,\mu,\tau)$ has strong mixing properties for dsh and H\"{o}lder continuous observables.\\
 
 {\bf{DSH Functions:}}
A function $\psi\in L^1(\pk)$ is called dsh if outside a pluripolar set $\psi=\varphi_1-\varphi_2$ where $\varphi_i$  are qpsh functions. This implies that $$dd^c\psi=T^+-T^-$$ for some positive closed $(1,1)$ currents $T^{\pm}$. Two dsh functions are identified if they coincide outside a pluripolar set; we denote the set of all dsh functions by $DSH(\pk)$. Note that dsh functions are stable under pull-back and push-forward operators induced by meromorphic self-maps of $\pk$ and have good compactness properties inherited from those of qpsh functions. Following \cite{DS3} one can define a norm on $DSH(\pk)$ as follows:
$$\|\psi\|_{DSH}:=\|\psi\|_{L^1(\pk)}+\inf\|T^{\pm}\|$$
where $dd^c\psi=T^+-T^-$ and the infimum is taken over all such representations. \\ \indent
 If $\mu$ is a probability measure on $\pk$ such that all qpsh functions are $\mu$-integrable then one can define 
 $$\|\psi\|_{DSH}^{\mu}:=|\la\mu,\psi\ra|+\inf\|T^{\pm}\|$$
where $T^{\pm}$ as above.
The following proposition is proved in \cite{DS4} we state it here for convenience of the reader:

\begin{prop}\label{pro}
Let $\psi\in DSH(\pk)$ then there exists negative qpsh functions $\varphi_1,\varphi_2$ such that $\psi=\varphi_1-\varphi_2$ and $dd^c\varphi_i\geq -c\|\psi\|_{DSH}\omega$ where $c>0$ independent of $\psi$ and $\varphi_i$'s. Moreover, $|\psi|$ is also a dsh function and $\||\psi|\|_{DSH}\leq c \|\psi\|_{DSH}.$ 
\end{prop} 
\subsection{Fiberwise Mixing} 
In this section, we explore the speed of mixing over the ``fibers" of $\tau$. For fixed $\lambda\in \mathscr{A}$ each $f_n:=\lambda(n)\in \Hd$ induces a unitary operator 
$$U_{n}:L^2_{\mu_{\theta^{n+1}(\lambda)}}(\pk)\to L^2_{\mu_{\theta^n(\lambda)}}(\pk)$$
$$\varphi \to \varphi\circ f_n.$$
We denote the adjoint of this operator by 
$$\mathcal{L}_n: L^2_{\mu_{\theta^n(\lambda)}}(\pk) \to  L^2_{\mu_{\theta^{n+1}(\lambda)}}(\pk)$$
$$\mathcal{L}_{n}\psi(x)=d^{-k}\sum_{f_n(y)=x}\psi(y).$$

 \begin{prop}\label{mix}
Let $\lambda\in \mathscr{A}$ be fixed. If $\varphi \in L^p_{\mu_{\theta^n(\lambda)}}(\pk)$ and $\psi \in DSH( \pk)$ then 
$$ |\la \mu_{\lambda}, (\varphi\circ f_{n-1}\circ \dots \circ f_0) \psi\ra - \la\mu_{n,\lambda}, \varphi\ra \la \mu_{\lambda},\psi\ra| \leq Cd^{-n}||\varphi||_{L^p_{\mu_{\theta^n(\lambda)}}} \ ||\psi||_{DSH}^{\mu_{\lambda}}$$ 
where $C>0$ depends only on $\lambda$ and $p>1.$
\end{prop}
We need several preliminary lemmas to prove Proposition \ref{mix}. The next lemma is an improved version of \cite[Proposition 7]{dTh1} and it will be helpful in the sequel. In what follows, $C_{\lambda}$ denotes a constant which depends on $\lambda.$
\begin{lem}\label{dsh}
For  $\lambda\in \mathscr{A}$ there exists $C_{\lambda}>0$ such that 
$$\|\psi\|_{DSH}^{\mu_{\theta^n(\lambda)}}\leq C_{\lambda}\|\psi\|_{DSH}$$
 for every $n\in \Bbb{N}$ and $\psi\in DSH(\pk).$
\end{lem}
\begin{proof}
Let $\psi \in DSH(\pk)$ then by Proposition \ref{pro} there exists qpsh functions $\phi_i$ such that $\psi=\phi_1-\phi_2$ and $dd^c\phi_i \geq -c\|\psi\|_{DSH}\omega$ where $c>0$ is independent of $\psi$ and $\phi_i$. Since random Green currents have H\"{o}lder continuous super potentials with H\"{o}lder exponent $0<\beta\leq 1,$ by Remark \ref{uniform} and \cite[ Lemma 3.5]{DN12} we obtain 
$$|\la \mu_{\theta^n(\lambda)},\psi\ra|\leq C_{\lambda}\max(\|\psi\|_{L^1},c^{1-\beta}\|\psi\|_{DSH}^{1-\beta}\|\psi\|_{L^1}^{\beta})$$
 If $\|\psi\|_{L^1}\geq c^{1-\beta}\|\psi\|_{DSH}^{1-\beta}\|\psi\|_{L^1}^{\beta}$ we are done. If not then $\|\psi\|_{L^1}<c\|\psi\|_{DSH}$ and this implies that 
$$|\la\mu_{\theta^n(\lambda)},\psi\ra|\leq cC_{\lambda}\|\psi\|_{DSH}$$

\end{proof}
\begin{rem}\label{rem}
By using a similar argument and using Lemma \ref{dsh} one can also show that for $\lambda\in \mathscr{A}$ there exists a constant $C_{\lambda}>0$  such that 
$$\|\psi\|_{DSH}\leq C_{\lambda}\max\|\psi\|_{DSH}^{\mu_{\theta^n(\lambda)}}.$$
for every $n\geq 0$ (cf. \cite[Proposition 8]{dTh1}). 
\end{rem}
The following lemma is essentially due to \cite{DNS}, however, we need to make some modifications to adapt it in our setting.
\begin{lem}\label{compact}
Let $\lambda\in \mathscr{A}$ and $\psi\in DSH(\pk)$. If $\la \mu_{\lambda},\psi\ra=0$ then there exists $C_{\lambda}>0$ such that for every $q\geq1$  
\begin{equation}\label{strong}
\| |\mathcal{L}_{n-1}\circ \dots \circ \mathcal{L}_1\circ \mathcal{L}_0(\psi)|\|_{L^q{\mu_{\theta^n(\lambda)}}}\leq qC_{\lambda} d^{-n} \|\psi\|_{DSH}^{\mu_{\lambda}}
\end{equation}
for $n\geq 1.$ 
\end{lem}
\begin{proof}
Note that $(f_{n-1})^*\mu_{\theta^n(\lambda)}=d^k\mu_{\theta^{n-1}(\lambda)}$ for $n\geq1$. This implies that 
$$\la\mu_{\theta^n(\lambda)},\mathcal{L}_{n-1}(\mathcal{L}_{n-2}\circ \dots \circ \mathcal{L}_0\psi)\ra=0.$$ Moreover, 
$$\|\mathcal{L}_{n-1}(\mathcal{L}_{n-2}\circ \dots \circ \mathcal{L}_0\psi)\|_{DSH}^{\mu_{\theta^n(\lambda)}}\leq d^{-1}\|\mathcal{L}_{n-2}\circ \dots \circ \mathcal{L}_0\psi\|_{DSH}^{\mu_{\theta^{n-1}(\lambda)}}$$ 
Indeed, we may write $$dd^c(\mathcal{L}_{n-2}\circ \dots \circ\mathcal{L}_0\psi)=R_{n-2}^+-R_{n-2}^-$$ where $R_{n-2}^{\pm}$ are some positive closed $(1,1)$ currents. 
Then
\begin{equation}\label{normm}
\|\mathcal{L}_{n-1}(\mathcal{L}_{n-2}\circ \dots \circ \mathcal{L}_0\psi)\|^{\mu_{\theta^n(\lambda)}}_{DSH}\leq\|d^{-k}(f_{n-1})_*(R_{n-2}^{\pm})\|= d^{-1}\|R_{n-2}^{\pm}\|
\end{equation} where the last equality follows from cohomological computation. Now by Proposition \ref{pro}, Remark \ref{rem}, Lemma \ref{dsh} and (\ref{normm}) we obtain 
\begin{eqnarray*}
\| |\mathcal{L}_{n-1}\circ \dots \circ \mathcal{L}_1\circ \mathcal{L}_0(\psi)|\|_{DSH} & \leq & C \| \mathcal{L}_{n-1}\circ \dots \circ \mathcal{L}_1\circ \mathcal{L}_0(\psi)\|_{DSH} \\
& \leq & C_1  \| \mathcal{L}_{n-1}\circ \dots \circ \mathcal{L}_1\circ \mathcal{L}_0(\psi)\|_{DSH}^{\mu_{\theta^n(\lambda)}} \\
& \leq & C_2 d^{-n}\|\psi\|_{DSH}^{\mu_{\lambda}}\\
& \leq & C_3d^{-n} \|\psi\|_{DSH}
\end{eqnarray*}
where $C_3>0$ depends on $\lambda$ but does not depend on $n$ nor $\psi.$ Thus, by above estimate and Lemma \ref{dsh} it is enough to prove the case $\|\psi\|_{DSH}>0$. Since $$\frac{d^n}{\|\psi\|_{DSH}}|\mathcal{L}_{n-1}\circ \dots \circ \mathcal{L}_1\circ \mathcal{L}_0(\psi)|$$ is a bounded sequence in $DSH(\pk)$ by Corollary \ref{mod} there exists $\beta>0$ and $C_{\lambda}>0$ such that
$$\la\mu_{\theta^n(\lambda)}, \exp(\beta \frac{d^n}{\|\psi\|_{DSH}}|\mathcal{L}_{n-1}\circ \dots \circ \mathcal{L}_1\circ \mathcal{L}_0(\psi)|)\ra \leq C_{\lambda}.$$
Now, by using the inequality $\frac{x^q}{q!}\leq e^x$ for $x\geq 0$ we conclude that
$$ \| |\mathcal{L}_{n-1}\circ \dots \circ \mathcal{L}_1\circ \mathcal{L}_0(\psi)|\|_{L^q_{\mu_{\theta^n(\lambda)}}}\leq \frac{q}{\beta}C_{\lambda}d^{-n}\|\psi\|_{DSH} $$ 
\end{proof}

\begin{proof}[Proof of Proposition \ref{mix}]
Let $\lambda\in \mathscr{A}$ be fixed. If $\psi$ is constant then the assertion follows from the invariance properties $$ (f_j)_*\mu_{\theta^j(\lambda)}=\mu_{\theta^{j+1}(\lambda)}.$$ Thus, replacing $\psi$ by  $\psi-\la\mu_{\lambda},\psi\ra$ we may assume that $\la\mu_{\lambda},\psi\ra=0.$ Then by H\"older's inequality and applying Lemma \ref{compact} with $q=\frac{p}{p-1}$ we obtain
\begin{eqnarray*}
|\la \mu_{\lambda}, (\varphi\circ f_{n-1}\circ \dots \circ f_0) \psi\ra| & = & d^{-kn}|\la F_{\lambda,n}^*\mu_{\theta^n(\lambda)}, (\varphi\circ f_{n-1}\circ \dots \circ f_0) \psi\ra| \\
& \leq & |\la \mu_{\theta^n(\lambda)},\varphi \mathcal{L}_{n-1}\circ \dots \circ \mathcal{L}_1\circ \mathcal{L}_0(\psi)\ra| \\
& \leq & \|\varphi\|_{L^p(\mu_{\theta^n(\lambda)})} \||\mathcal{L}_{n-1}\circ \dots \circ \mathcal{L}_1\circ \mathcal{L}_0(\psi)|\|_{L^q(\mu_{\theta^n(\lambda)})}\\
& \leq & \frac{p}{p-1}C_{\lambda} d^{-n}\|\varphi\|_{L^p(\mu_{\theta^n(\lambda)})} \|\psi\|_{DSH}^{\mu_{\lambda}}
\end{eqnarray*}
for some $c>0$ independent of $\psi$ and for all $n\geq 0.$ 
\end{proof}
In the deterministic case, as a consequence of interpolation theory between the Banach spaces $\mathscr{C}^0$ and $\mathscr{C}^2$ \cite{Triebel}; it was observed in \cite{DNS} that a holomorphic map $f\in\Hd$ posses strong mixing property for $\beta$-H\"{o}lder continuous functions with $0<\beta\leq 1$ (see \cite[Proposition 3.5]{DNS}). Adapting their argument to our setting, we obtain the succeeding lemma. We omit the proof as it is similar to the one given in Lemma \ref{compact} and to that of \cite[Proposition 3.5]{DNS}.
\begin{lem}\label{hholder}
Let $\lambda\in\mathscr{A}$ be fixed, $q>1$ and $0<\beta\leq1.$ If $\psi:\pk\to \Bbb{R}$ be a $\beta$-H\"older continuous function such that $\la\mu_{\lambda},\psi\ra=0$ then there exists a constant $C_{\lambda,\beta}>0$ such that
$$ \| |\mathcal{L}_{n-1}\circ \dots \circ \mathcal{L}_1\circ \mathcal{L}_0(\psi)|\|_{L^q(\mu_{\theta^n(\lambda)})}\leq C_{\lambda,\beta} d^{-\frac{n\beta}{2}} \|\psi\|_{\mathscr{C}^{\beta}}
$$
for every $n\geq1.$
\end{lem}
\subsection{Exponential Mixing}\label{expo}
In this section we prove that the dynamical system $(X,\mathscr{B},\tau,\mu)$ is exponentially mixing for dsh and H\"{o}lder continuous observables. We let $\pi:X\to\pk$ be the projection on the second factor. For a measurable function $\varphi:X\to \Bbb{R}$ we denote $\varphi_{\lambda}(x):=\varphi(\lambda,x)$ and for a measurable function $\psi:\pk\to\Bbb{R}$ we define $\tilde{\psi}:=\psi \circ \pi$. Note that
\begin{equation}\label{til}
\|\tilde{\psi}\|^2_{L^2_{\mu}(X)}=\int_{\Omega}\la\mu_{\lambda},|\psi|^2\ra d\p(\lambda)=\int_{\Omega}\|\psi\|_{L^2_{\mu_{\lambda}}}^2d\p(\lambda)
\end{equation} 
 Let us denote  the unitary operator induced by $\tau$  
 $$U_{\tau}:L^2(X)\to L^2(X)$$ 
 $$\varphi\to \varphi\circ \tau$$
 and $P_{\tau}=U_{\tau}^*$ is the adjoint operator.
 
 \begin{prop}\label{adjoint}
 Let $\psi:\pk\to\Bbb{R}$ and $\psi\in L_{\nu}^2(\pk)$ then 
 $$P_{\tau}\tilde{\psi}(\lambda,x)=\int_{\pN}\mathcal{L}_f\psi(x) dm(f)$$
 \end{prop}
 \begin{proof}
 Let $\varphi\in L^2_{\mu}(X).$ We denote $\lambda':=\theta(\lambda).$ Then by (\ref{inv}) and using the fact that $\p$ is the product measure and by Fubini's theorem we obtain
 \begin{eqnarray*}
 \la U_{\tau}\varphi,\tilde{\psi}\ra & = & \int_{\Omega}\la\mu_{\lambda}(dx), \varphi_{\theta(\lambda)}(f_0x) \psi(x) \ra d\p(\lambda) \\
 & = & \int_{\Omega} \la d^{-k}(f_0)^*\mu_{\theta(\lambda)}(dx),\varphi_{\theta(\lambda)}(f_0x)\psi(x)\ra d\p(\lambda) \\
 & = & \int_{\Omega}\int_{\pN}\la \mu_{\lambda'}(dx),\varphi_{\lambda'}\mathcal{L}_f\psi(x)\ra dm(f) d\p(\lambda')\\
& = & \int_{\Omega}\la \mu_{\lambda'},\varphi_{\lambda'}\int_{\pN}\mathcal{L}_f\psi dm(f)\ra d\p(\lambda')\\
& = & \la\varphi,P_{\tau}\tilde{\psi}\ra
 \end{eqnarray*}
where the third line follows from the invariance of $\p$ under $\theta$. 
 \end{proof}
For $\varphi,\psi\in L^2_{\mu}(X)$ we define the correlation function by 
$$C_n(\varphi,\psi):=\la\mu,\varphi\circ \tau^n\psi\ra-\la\mu,\varphi\ra\la\mu,\psi\ra.$$
Note that the dynamical system $(X,\mathscr{B},\mu,\tau)$ is mixing if $C_n(\varphi,\psi)\to0$ as $n\to \infty$ for every $\varphi,\psi\in L^2_{\mu}(X).$ Next, we prove that $C_n$ decays exponentially fast for dsh observables. 
\begin{proof}[Proof of Theorem \ref{mixx}]
Again by the invariance properties (\ref{inv}) without lost of generality we may assume that $\la\mu,\tilde{\psi}\ra=\int_{\Omega}\la\mu_{\lambda},\psi\ra d\p(\lambda)=0.$ Since $\psi$ is real valued we have
\begin{equation*}
C_n(\varphi,\tilde{\psi})=  \la \mu, \varphi P_{\tau}^n(\tilde{\psi})\ra 
\end{equation*} 
hence, we need to bound the quantity $|\la\mu,\varphi P_{\tau}(\tilde{\psi})\ra|.$\\ \indent
Now, denote by $\mathcal{L}_n:=\mathcal{L}_{f_n}$ by a straightforward calculation and $\la\mu,\tilde{\psi}\ra=0$ we see that 
 \begin{eqnarray*}
 P_{\tau}^n\tilde{\psi}(\lambda,x) & = & \int_{\Omega} \mathcal{L}_{n-1} \circ \dots \circ \mathcal{L}_0 \psi(x) d\p(\lambda)\\
 & = & \int_{\Omega} (\mathcal{L}_{n-1} \circ \dots \circ \mathcal{L}_0 \psi(x)-\la \mu_{\lambda},\psi\ra) d\p(\lambda)\
 \end{eqnarray*}

 On the other hand, since supp$(m)\subset\Hd$ the function $\log dist(\cdot,\mathcal{M})$ is bounded. Then by Remark \ref{uniform}, Corollary \ref{mod}, Lemma \ref{compact} and Lemma \ref{dsh} for $\p$-a.e. $\lambda\in \Omega$ we have
$$\|\mathcal{L}_{n-1} \circ \dots \circ \mathcal{L}_0 \psi-\la\mu_{\lambda},\psi\ra\|_{L^2_{\mu_{\theta^n(\lambda)}}}\leq Cd^{-n}\|\psi\|_{DSH} $$
where $C>0$ does not depend on $\lambda$ or $n.$ Then by H\"{o}lder's inequality, (\ref{til}) and from above argument we infer that 
\begin{eqnarray*}
|C_n(\varphi,\tilde{\psi})| & = & |\la \mu,\varphi P_{\tau}^n(\tilde{\psi})\ra|\\
& \leq & \|\varphi\|_{L^2_{\mu}} \|P_{\tau}^n(\tilde{\psi})\|_{L^2_{\mu}} \\
& \leq &  \|\varphi\|_{L^2_{\mu}} (\int_{\Omega} \|P_{\tau}^n(\tilde{\psi})\|^2_{L^2_{\mu_{\lambda}}} d\p(\lambda))^{\frac12} \\
& = &  \|\varphi\|_{L^2_{\mu}} (\int_{\Omega} \|P_{\tau}^n(\tilde{\psi})\|^2_{L^2_{\mu_{\theta^n(\lambda)}}} d\p(\lambda))^{\frac12} \\
& \leq &  C d^{-n}\|\varphi\|_{L^2_{\mu}}\|\psi\|_{DSH}
\end{eqnarray*}
where the forth equality follows from $\theta_*\p=\p$ and $P^n_{\tau}\tilde{\psi}$ does not depend on $\lambda$.

 \end{proof} 
\begin{rem}
Note that we can also obtain exponential decay of correlations for H\"{o}lder continuous functions by using Lemma \ref{hholder} and applying the above argument. 
\end{rem}

\section{ Stochastic Properties of $(X,\mathscr{B},\mu,\tau)$}
 
 \subsection{Central Limit Theorem} 
 In this section we prove a Central Limit theorem (CLT) for d.s.h and H\"{o}lder continuous observables. Our proof relies on verifying Gordin's condition.\\ 
 
 {\bf{Gordin's Method:}} Let $(X,\mathscr{F},T,\alpha)$ be an ergodic dynamical system. We let $$U:L^2_{\alpha}(X)\to L^2_{\alpha}(X)$$ $$\phi\to\phi\circ T$$denote the unitary operator induced by $T$ and let $P:=U^*$ be its adjoint operator. We denote the $\sigma$-algebra $\mathscr{F}_n:=T^{-n}(\mathscr{F})$ and let $E(\cdot|\mathscr{F}_n)$ be the associated conditional expectation. Recall that $E(\phi|\mathscr{F}_n)$ is the orthogonal projection of $\phi\in L^2_{\alpha}(X)$ onto closed subspace of $\mathscr{F}_n$ measurable functions in $L^2_{\alpha}(X).$ Then it follows from an easy calculation that for $n\geq0$
$$\|E(\phi|\mathscr{F}_n)\|_{L^2_{\alpha}}=\|P^n\phi\|_{L^2_{\alpha}}\ \text{and}\ E(\phi|\mathscr{F}_n)=U^nP^n\phi$$ almost everywhere with respect to $\alpha$ restricted to $\mathscr{F}_n.$ We say that $\psi\in L^2_{\alpha}(X)$ is a \textit{coboundary} if $\psi=u\circ T-u$ for some $u\in L^2_{\alpha}(X).$ In the sequel, we let $\mathcal{N}(0,\sigma)$ denote the normal distribution with mean zero and variance $\sigma>0$.

\begin{thm}\cite{Gordin}

Let $\phi\in L^2_{\alpha}(X)$ be such that $\la\alpha,\phi\ra=0.$ Assume that
\begin{equation}\label{dual}
\sum_{n\geq0}\|P^n\phi\|_{L^2_{\alpha}}<\infty
\end{equation}
then the non-negative real number $\sigma$ defined by 
$$\sigma^2=\lim_{N\to \infty}\frac{1}{N}\int_X( \sum_{n=0}^{N-1}\phi\circ T^n)^2d\alpha$$
is a finite number. Moreover, $\sigma>0$ if and only if $\phi$ is not a coboundary. In this case,
$$\frac{1}{\sqrt{N}}\sum_{n=0}^{N-1}\phi\circ T^n \Rightarrow \mathcal{N}(0,\sigma).$$ 
as $N\to \infty.$
\end{thm}
 
 \begin{proof}[Proof of Theorem \ref{CLT}]
 We will verify condition (\ref{dual}). To this end, it is enough to show that 
 $$\sum_{j\geq0}\| P_{\tau}^j\tilde{\psi}\|_{L^2_{\mu}} <\infty $$
where $P_{\tau}$ is as defined in section \ref{expo}. Then by Proposition \ref{adjoint} we have 
$$P_{\tau}^j\tilde{\psi}(x) = \int_{\Omega} \mathcal{L}_{j-1} \circ \dots \circ \mathcal{L}_0 \psi(x) d\p(\lambda).$$
 Now, since $\la\mu,\tilde{\psi}\ra=0$ the argument in the proof of Theorem \ref{mixx} yields
 \begin{eqnarray*}
 \|P_{\tau}^j\tilde{\psi}\|_{L^2_{\mu}} & \leq &   Cd^{-j} \|\psi\|_{DSH}
 \end{eqnarray*}
 thus, the assertion follows.
 \end{proof}
\begin{rem}
Note that applying the same reasoning and using Lemma \ref{hholder}, one can obtain CLT for H\"{older} continuous functions. 
\end{rem}

 \section{A Markov chain associated with random pre-images} \label{markov}
 In this section we introduce a Markov chain associated with pre-images of random holomorphic maps. We use the same notation as in previous sections. We consider $(\pk,\mathcal{B}, \nu)$ as a probability space where $\mathcal{B}$ denotes the Borel algebra, $\nu:=\pi_*\mu$ and $\pi:\Omega\times\pk \to \pk$ is the projection on the second factor. We let $\mathcal{L}_f$ denote the Perron-Frobenius operator associated with $f\in\Hd,$ precisely  
$$\mathcal{L}_f(\phi)(x)=d^{-k}\sum_{\{y:f(y)=x\}}\phi(y)$$ for $\phi\in L^2_{\nu}(\pk).$  We define the transition probability by 
$$P:\pk\times \mathcal{B}\to [0,1]$$
\begin{eqnarray*} \label{trans}
P(x,G): & = &   \int_{\pN}\mathcal{L}_f(\chi_G)(x)dm(f)\\
& = & \int_{\pN}d^{-k} \sum_{y\in f^{-1}(x)}\delta_y(G) dm(f) 
\end{eqnarray*} 
where $\delta_y$ denotes the Dirac mass at $y$ and $\chi_G$ denotes the indicator function of $G.$ First, we observe that $P(x,G)$ is well-defined. To this end it is enough to show that for fixed $G\in \mathcal{B}$ and $x\in\pk$ the map
$$\Hd\to [0,1]$$
$$f\to \mathcal{L}_f(\chi_G)(x)$$ is measurable. 
This follows from noting that for fixed $x\in \pk$ as $f$ varies in $\Hd$ the solutions $y\in \pk$ such that $f(y)=x$ vary continuously.  The same reasoning shows that $x \to P(x,G)$ is a measurable map for every Borel set $G.$ Moreover, $G\to P(x,G)$ defines a probability on $\pk.$
Thus, we may define the Markov operator on non-negative measurable functions by
\begin{eqnarray*}
P\phi(x) & := & \int_{\pk}\phi(y)P(x,dy)\\
& = &  \int_{\pN}\mathcal{L}_f\phi(x)dm(f)
\end{eqnarray*}
which is again a non-negative measurable function. The following is a direct consequence of Theorem \ref{main} and Theorem \ref{mixx}:
\begin{prop}\label{station}
The measure $\nu$ is an $P$-invariant ergodic measure.
\end{prop}
\begin{proof}
To prove invariance, we need to show that for every bounded measurable function $\phi$ on $\pk$ we have $$\la\nu,P\phi\ra=\la\nu,\phi\ra.$$
We denote $\lambda:=\theta(\lambda').$ Then by definition of $\nu$ and Fubini's theorem we have 
\begin{eqnarray*}
\la\nu,P\phi\ra & = & \la\mu,P\phi\circ \pi\ra \\
& = & \int_{\Omega}\la\mu_{\lambda},P\phi\ra d\p(\lambda) \\
& = & \int_{\Omega}\int_{\pN} \la\mu_{\lambda},\mathcal{L}_f\phi\ra dm(f) d\p(\lambda) \\
& = & \int_{\Omega} \la\mu_{\theta(\lambda')},\phi\ra d\p(\lambda')\\
& = & \la \mu, \phi\circ \pi\ra\\
& = & \la \nu,\phi\ra.
\end{eqnarray*}
To prove ergodicity of $\nu$ we need to show that for every bounded measurable function $\phi$ on $\pk,$ $P\phi=\phi$ implies that $\phi$ is constant $\nu$-a.e. equivalently $\phi\circ \pi$ is constant $\mu$-a.e. This follows from Proposition \ref{adjoint} and the strong mixing property proved in Theorem \ref{mixx}. 
\end{proof}

Let $Y,\p_{\nu},(Z_n)_{n\geq0}$ and $\vartheta:Y\to Y$ be as defined in the introduction. It follows from Proposition \ref{station} that $\p_{\nu}$ is invariant and ergodic with respect to the shift $\vartheta$ hence, $(Z_n)_{n\geq 0}$ is stationary under $\p_{\nu}.$ Thus, by  Birkhoff's ergodic theorem, for every $\phi\in L^1_{\nu}(\pk)$ the series  
$$\frac{1}{N}\sum_{n=0}^{N-1} \phi(Z_n(y))=\frac{1}{N}\sum_{n=0}^{N-1}\phi(Z_0\circ \vartheta^n(y))\ \text{converges to}\ \la\nu,\phi\ra$$ as $N\to \infty$ for $\p_{\nu}$-a.e. $y\in Y.$ We say that $\phi$ \textit{satisfies Central Limit Theorem} (CLT) for the Markov chain $(Z_n)_{n\geq0}$ if 
$\frac{1}{\sqrt{N}}\sum_{n=0}^{N-1} \phi(Z_n)$ converges in law under the invariant measure $\Bbb{P}_{\nu}$ to the normal distribution $\mathcal{N}(0,\sigma)$ for some $\sigma>0.$    
The following result is a consequence of \cite{GL}:
\begin{thm}\label{GL}
 If  $\psi=g-Pg$ for some $g\in L^2_{\nu}(\pk)$ then $$\frac{1}{\sqrt{N}}\sum_{n=0}^{N-1}\psi(Z_n)\Rightarrow \mathcal{N}(0,\sigma^2)$$ where
$\sigma^2=\int_{\pk}g^2-(Pg)^2d\nu.$ Moreover, if $\sigma=0$ then the partial sum converges to Dirac mass at 0. 
\end{thm}

\begin{proof}[Proof of Theorem \ref{CLTM}]
Note that the hypothesis in the Theorem \ref{GL} is satisfied if $$ \sum_{j\geq0}\|P^j \psi \|_{L^2_{\nu}} < \infty.$$ Indeed, if $g:=\sum_{j\geq0}P^j\psi$ converges in $L^2_{\nu}(\pk)$ then $\psi=g-Pg.$\\
 Now, by $\la\nu,\psi\ra=0$ we have 
 \begin{eqnarray*}
P^j\psi(x) & = & \int_{\Omega}\mathcal{L}_{j-1}\circ\dots \circ \mathcal{L}_0(\psi)(x)d\p(\lambda)\\
&  = & \int_{\Omega}\mathcal{L}_{j-1}\circ\dots \circ \mathcal{L}_0(\psi)(x)-\la\mu_{\lambda},\psi\ra d\p(\lambda)
\end{eqnarray*}
Thus, by Remark \ref{uniform}, Corollary \ref{mod}, Lemma \ref{compact} and Lemma \ref{dsh} there exists $C>0$ such that 
$$\| \mathcal{L}_{j-1}\circ \dots \circ \mathcal{L}_1\circ \mathcal{L}_0(\psi)-\la\mu_{\lambda},\psi\ra\|_{L^2_{\mu_{\theta^j(\lambda)}}}\leq C d^{-j} \|\psi\|_{DSH}$$
for $\p$-a.e $\lambda\in \Omega.$ On the other hand by invariance property $\theta_*\p=\p$

\begin{eqnarray*}
\|P^j\psi\|^2_{L^2_{\nu}} & = & \int_{\Omega} \|P^j\psi\|^2_{L^2_{\mu_{\lambda}}}d\p(\lambda) \\
& = & \int_{\Omega} \|P^j\psi\|^2_{L^2_{\mu_{\theta^j(\lambda)}}}d\p(\lambda)\\
& \leq & Cd^{-j}\|\psi\|_{DSH}
\end{eqnarray*}
where the second line follows from $\theta_*\p=\p$ and $P^j\psi$ does not depend on $\lambda.$
\end{proof}

\begin{rem}
Applying the same argument and using Lemma \ref{hholder} one can obtain CLT for H\"{o}lder continuous observables.
\end{rem}


 \bibliographystyle{amsalpha}

\bibliography{biblio}
\end{document}